\documentclass[11pt]{article}

\usepackage{paralist, amsmath, amsthm, amssymb, color}
\usepackage[margin=1in]{geometry}
\usepackage{tikz}
\tikzstyle{wv}=[circle,draw=black!90,fill=white!20,thick,inner sep=2pt,minimum width=5pt] 
\tikzstyle{bv}=[circle,draw=black!90,fill=black!100,thick,inner sep=2pt,minimum width=5pt] 
\tikzstyle{gv}=[circle,draw=black!70,fill=gray!100,thick,inner sep=2pt,minimum width=5pt]

\newtheorem{theorem}{Theorem}

\newtheorem{definition}{Definition}
\newtheorem{lemma}{Lemma}
\newtheorem{proposition}{Proposition}

\newtheorem{example}{Example}

\newtheorem{remark}{Remark}

\newcommand{\la}{\lambda}
\newcommand{\cycty}{cyctype}
\newcommand{\ty}{type}

\newenvironment{theorem*}{\medskip\noindent {\bf Theorem} \it}{\medskip}

\def\ie{{\it i.e.\ }}

\def\eg{{\it e.g.\ }}

\title{Bijective evaluation of the connection coefficients of the double coset algebra}

\author{Alejandro H. Morales and Ekaterina A. Vassilieva}

\begin{document}

\maketitle

\begin{abstract}
This paper is devoted to the evaluation of the generating series of the connection coefficients of the double cosets of the hyperoctahedral group. Hanlon, Stanley, Stembridge (1992) showed that this series, indexed by a partition $\nu$, gives the spectral distribution of some random real matrices that are of interest in random matrix theory. We  provide an explicit evaluation of this series when $\nu=(n)$ in terms of monomial symmetric functions. Our development relies on an interpretation of the connection coefficients in terms of locally orientable hypermaps and a new bijective construction between locally orientable partitioned hypermaps and some permuted forests.
\end{abstract}

%\clearpage
%\tableofcontents

\section{Introduction}

In what follows, we denote by $\la=(\la_1,\la_2,\ldots,\la_k)_{\geq} \vdash n$ an integer partition of $n$ and $\ell(\la)=k$ the number of parts of $\la$. 
%Thus, $\la=(\la_1,\ldots, \la_k)$ where $\la_1 \geq \cdots \geq \la_k\geq 1$ and $\sum \la_i=n$. 
If $n_i(\la)$ is the number of parts of $\la$ that are equal to $i$ (by convention $n_0(\la)=0$), then we write $\la $ as $1^{n_1(\la)}\,2^{n_2(\la)}\ldots$ and let $Aut(\la)=\prod_i n_i(\la)!$. Also, if $\la\vdash n$, let $\la\la$ and $2\la$ be the partitions of $2n$ $(\la_1,\la_1,\la_2,\la_2,\ldots)$ and $(2\la_1,2\la_2,\ldots)$ respectively. Let $[m]=\{1,\ldots,m\}$ and $S_m=S_m([m])$ be the symmetric group on $m$ elements, and let $\mathcal{C}_{\la}$ be the conjugacy class in $S_m$ of permutations $w$ with cycle type, $\cycty(w)$, $\la\vdash m$. 

We look at {\bf perfect pairings} of the set $[n]\cup [\widehat{n}]=\{1,\ldots n, \widehat{1},\ldots, \widehat{n}\}$ of {\bf non-hat} or {\bf hat} numbers which we view as fixed point free involutions in $S_{2n}([n]\cup [\widehat{n}])$. Note that the disjoint cycles of the product $f\circ g$ have repeated lengths \ie, $f\circ g \in \mathcal{C}_{\la\la}$. Also, given $w\in S_{2n}$, let $f_w$ be the pairing $(w^{-1}(1),w^{-1}(\hat{1}))\cdots(w^{-1}(n),w^{-1}(\hat{n}))$.  

Let $B_n$ be the hyperoctahedral group which we view as the centralizer in $S_{2n}$ of the involution $f_{\star}=(1\widehat{1})(2\widehat{2})\cdots(n\widehat{n})$. Then $|B_n| = 2^nn!$, and it is well known that the double cosets of $B_n$ in $S_{2n}$ are indexed by partitions $\nu$ of $n$ and consist of permutations $w\in S_{2n}$ such that the cycle type of $f_{\star}\circ f_w$ is $\nu\nu$ \cite[Ch. VII.2]{M}. If we denote such a double coset by $K_{\nu}$ and pick from it a fixed element $w_{\nu}$, then let $b_{\la,\mu}^{\nu}$ be the number of ordered factorizations $u_1\cdot u_2$ of $w_{\nu}$ where $u_1\in K_{\la}$ and $u_2\in K_{\mu}$. \ie,
\[
b_{\la,\mu}^n = ~\mid \{ (u_1,u_2) \mid u_1\cdot u_2 = w_{\nu}, u_1 \in K_{\la}, u_2 \in K_{\mu} \}  \mid. 
\]
We provide a combinatorial formula for $b_{\la,\mu}^{\nu}$ when $\nu=(n)$  (say $w_{(n)}=(123\ldots n)(\widehat{n}\,n\widehat{-}1\,n\widehat{-}2\ldots \widehat{1})$) by interpreting these factorizations as {\bf locally orientable unicellular  (partitioned) hypermaps}. 

We give this formula using symmetric functions: for $\la \vdash n$, we use the monomial symmetric function $m_{\la}({\bf x})$ which is the sum of all different monomials obtained by permuting the variables of $x_1^{\la_1}x_2^{\la_2}\cdots$, and the power symmetric function $p_{\lambda}({\bf x})$, defined multiplicatively as $p_{\la}=p_{\la_1}p_{\la_2}\cdots$ where $p_n({\bf x})=m_n({\bf x})=\sum_{i} x_i^n$.

%Let $f_1=(1\,\widehat{n})(2\,\widehat{1})(3\,\widehat{2})\ldots(n\,n\widehat{-}1)$ and $f_2 = (1\,\widehat{1})(2\,\widehat{2})\ldots(n\,\widehat{n})$.
%Then $f_1\circ f_2 = (123\ldots n)(\widehat{n}n\widehat{-}1\, n\widehat{-}2\ldots\widehat{1}) \in \mathcal{C}_{(n),(n)}$. For $\la,\mu \vdash n$ let 
%\begin{equation}
%L_{\lambda, \mu}^{n} =~ \mid\{ \mbox{ pairings } f_3 \mbox{ pairings in } S_{2n} \mbox{ such that } f_3\circ f_1 \in \mathcal{C}_{\lambda\lambda} \mbox{ and } f_3\circ f_2 \in \mathcal{C}_{\mu\mu} \}\mid
%\end{equation}
%These numbers count {\bf locally orientable unicellular hypermaps} which are certain bipartite graphs embedded on locally orientable surfaces which are generalizations of orientable unicellular maps that are counted in the Harer Zagier formula \cite{HZ} (see Section \ref{}). In addition, the numbers $2^nn!L^n_{\la,\mu}$ are certain {\bf connection coefficients} of the double coset algebra of the hyperoctahedral group \cite{HSS} and \cite{GJ1} (see Section \ref{}). 

%The main combinatorial result of this text is the following formula for $L^n_{\la,\mu}$.

\begin{theorem}\label{thm:main}
Let  $b_{\la,\mu}^{n}$ be the number of ordered factorizations $u_1\cdot u_2$ of $w_{(n)}$ where $u_1\in K_{\la}$ and $u_2\in K_{\nu}$. If $p_{\la}$ and $m_{\la}$ are the power and monomial symmetric functions then 
%Let  $L^n_{\la,\mu}$ be the number of pairings in $S_{2n}$ such that $f_3\circ f_1 \in \mathcal{C}_{\lambda\lambda}$ and $f_3\circ f_2 \in \mathcal{C}_{\mu\mu}$. If $p_{\la}({\bf x})$ and $m_{\la}({\bf x})$ are the power and monomial symmetric functions then 
\begin{multline} \label{eq:mainthm} 
\frac{1}{2^n n!}\sum_{\lambda,\mu \vdash n}b_{\lambda, \mu}^{n}\,p_{\lambda}p_{\mu}=\sum_{\lambda,\mu \vdash n}Aut(\lambda)Aut(\mu)m_{\lambda}m_{\mu} \sum_{{\bf A}\in M(\lambda,\mu)}\frac{\mathcal{N}({\bf A})}{{\bf A}!}\frac{(n-q-2r)!(n-p-2r)!}{(n+1-p-q-2r)!}\times\\
\times \frac{p'!q'!\left ( r - p' \right)!\left ( r- q' \right)!}{2^{2r-p'-q'}}\prod_{i,j,k}{\binom{i-1}{j,k,j+k}}^{(P+Q)(i,j,k)}{\binom{i-1}{j,k,j+k-1}}^{(P'+Q')(i,j,k)} 
\end{multline}
Where, $M(\lambda,\mu)$ is the set of $4$-tuples ${\bf A}=(P,P',Q,Q')$ of tridimensional arrays of non negative integers indexed by $i,j,k\geq 0$ 
%indexed by $i,j,k\geq 0$ 
with $p=| P| = \sum_{i,j,k\geq 0}P_{ijk} \neq 0,  p'=\ell(\la)-p=| P'|, q=| Q |$, $q'=\ell(\mu)-q=| Q' |$, and 
$$
\begin{array}{rclrcl}
n_i(\lambda)&=&\sum_{j,k}P_{ijk} + P'_{ijk}, & n_i(\mu) &=& \sum_{j,k}Q_{ijk} + Q'_{ijk},\\[2mm]
 r &=& \sum_{i,j,k}(j+k)(P_{ijk} + P'_{ijk}), & r&=&\sum_{i,j,k}(j+k)(Q_{ijk} + Q'_{ijk}),  \\[2mm]
 q' &=&\sum_{i,j,k}j(P_{ijk} + P'_{ijk}), & p' &=& \sum_{i,j,k}j(Q_{ijk} + Q'_{ijk}).\\
\end{array}
$$

where ${\bf A!} = \prod_{i,j,k}P_{ijk}!\,P'_{ijk}!\,Q_{ijk}!\,Q'_{ijk}!$ and if $q'\neq 0$,
\footnotesize
\begin{multline*}
\mathcal{N}({\bf A}) = \frac{1}{q'}\sum_{t,u,v}\frac{tP_{tuv}}{t-2u-2v}\left [(t-2b-2v)\left (\frac{\delta_{p'\neq 0}}{p'}\sum_{i,j,k}{jQ}\sum_{i,j,k}{jP'}+\frac{\sum_{i,j,k}{(i-1-2j-2k)Q}\sum_{i,j,k}{jP}}{n-q-2r}\right) + \right.   \\
+ u\left. \left (\frac{\delta_{p'\neq 0}}{p'}\sum_{i,j,k}{(i-2j-2k)P'}\sum_{i,j,k}{jQ'}+ \sum_{i,j,k}{(i-2j-2k)Q'}\frac{1+\sum_{i,j,k}{(i-1-2j-2k)P}}{n-q-2r}\right )\right], 
\end{multline*}
\normalsize
otherwise if $q'=0$, then $\mathcal{N}({\bf A})=\sum_{t,u,v}tP_{tuv}$.

\end{theorem}

\begin{remark}[special cases] There are special cases of the formula detailed in Appendix \ref{app}.
\end{remark}

\subsection{Background on connection coefficients $b^{\nu}_{\la,\mu}$}

By abuse of notation, let the double coset $K_{\nu}$ also represent the sum of its elements in the group algebra $\mathbb{C}S_{2n}$. Then $K_{\nu}$ form a basis of a commutative subalgebra of $\mathbb{C}S_{2n}$ (the {\em Hecke algebra} of the {\em Gelfand pair} $(S_{2n},B_n)$) and one can check that $K_{\la}\cdot K_{\mu}=\sum_{\nu} b^{\nu}_{\la,\mu} K_{\nu}$. Thus, $\{b_{\la,\mu}^{\nu}\}$ are the {\bf connection coefficients} of this double coset algebra.

We use $Z_{\la}({\bf x})$ to denote the {\bf zonal polynomial} indexed by $\la$ which can be viewed as an analogue of the Schur function $s_{\la}$ (for more information on these polynomials see \cite[Ch. VII]{M}). In terms of $p_{\mu}$: $s_{\la}= \sum_{\mu} z_{\mu}^{-1} \chi^{\la}_{\mu}p_{\mu}$ where  $z_{\la}= Aut(\la)\prod_i i^{n_i(\la)}$, $\chi^{\la}_{\mu}$ are the irreducible characters of the symmetric group; and  $Z_{\la}({\bf x}) = \frac{1}{|B_n|}\sum_{\mu \vdash n} \varphi^{\la}(\mu) p_{\mu}$ where $\varphi^{\la}(\mu)=\sum_{w\in K_{\mu}} \chi^{2\la}_{\cycty(w)}$. The following formula for the connection coefficients in terms of $\varphi^{\la}(\mu)$ was given in \cite[Lemma 3.3]{HSS}:
\begin{equation} \label{eq:conncoeffs}
b^{\nu}_{\la,\mu} = \frac{1}{|K_{\nu}|} \sum_{\beta \vdash n }\frac{1}{H_{2\nu}} \varphi^{\beta}({\nu})\varphi^{\beta}({\la}) \varphi^{\beta}({\mu}), 
\end{equation}
where $| K_{\nu} | = | B_n | | \mathcal{C}_{\nu} | 2^{n-\ell(\nu)}$ \cite[Lemma 2.1]{GJ1}, and $H_{2\la}$ is the product of all the {\em hook-lengths} of the partition $2\la$.

Let $\Psi^{\nu}({\bf x},{\bf y}) = \frac{1}{| B_n|}\sum_{\la,\mu}b^{\nu}_{\la,\mu} p_{\la}({\bf x})p_{\mu}({\bf y})$ so that $\Psi^{(n)}$ is the LHS of \eqref{eq:mainthm}. Equation \eqref{eq:conncoeffs} immediately implies that $\Psi_{\nu}({\bf x},{\bf y})= \frac{1}{| K_{\nu}|}\sum_{\la \vdash n} \frac{| B_n |}{H_{2\la}} \varphi^{\la}(\nu) Z_{\la}({\bf x}) Z_{\mu}({\bf y})$. Moreover, if for an $n\times n$ matrix $X$ we say that $p_{k}(X) =trace(X^{k})$, then in \cite[Thm. 3.5]{HSS} it was shown that $\Psi^{\nu}$ is also the expectation of $p_{\nu}(XUYU^T)$ over $U$, where $U$ are $n\times n$  matrices whose entries are independent standard normal random {\em real} variables and $X,Y$ are arbitrary but fixed real symmetric matrices ($\phantom{~}^T$ denotes transpose). Note that the quantities $trace((XUYU^T)^k)$ determine the eigenvalues of $XUYU^T$.

\section{Combinatorial formulation}
\subsection{Unicellular locally orientable hypermaps}
From a topological point of view, {\bf locally orientable hypermaps} of n edges can be defined as a connected bipartite graph with black and white vertices. We view each edge as being a {\em fat} edge composed of two edge-sides both connecting the two incident vertices (\begin{tikzpicture}[scale=0.75]
\node (4) at (0,0) [wv] {};
\node (5) at (1,0) [bv] {};
\node (0) at (0,-0.1) {};
\node (1) at (1,-0.1) {};
\node (2) at (0,0.1) {};
\node (3) at (1,0.1) {};
\draw [-] (0) -- (1);
\draw [-] (2) -- (3);
\end{tikzpicture}). We call these edge-sides {\bf half edges}. This graph is embedded on a locally orientable surface such that if we cut the graph from the surface the remaining part consists of connected components called  faces or cells, each homomorphic to an open disk. Note that two half edges can be parallel or cross in the middle. A  crossing or a twist of two half edges  (\begin{tikzpicture}[scale=0.75]
\node (4) at (0,0) [wv] {};
\node (5) at (1,0) [bv] {};
\node (0) at (0,-0.1) {};
\node (1) at (1,-0.1) {};
\node (2) at (0,0.1) {};
\node (3) at (1,0.1) {};
\draw [-] (0) -- (3);
\draw [-] (2) -- (1);
\end{tikzpicture}) indicates a change of orientation in the map and that the map is embedded on a non orientable surface (projective plan, Klein bottle, Moebius strip,...).  
%Graphical examples of locally orientable hypermaps can be found in \cite{FS}.
In \cite{GJ1}, it was shown that such hypermaps admit a natural formal description involving triples of perfect pairings $(f_1, f_2, f_3)$ on the set of half edges  where:
\begin{compactitem}
\item $f_3$ associates half edges of the same edge,
\item $f_1$ associates immediately successive (i.e. with no other half edges in between) half edges moving around the white vertices, and 
\item $f_2$ associates immediately successive half edges moving around the black vertices.
\end{compactitem}

Formally we label each half edge with an element in $[n]\cup [\widehat{n}]=\{1,\ldots,n,\widehat{1},\ldots,\widehat{n}\}$ and define $(f_1, f_2, f_3)$ as perfect pairings on this set.
Combining the three pairings gives the fundamental characteristics of the hypermap since:
\begin{compactitem}
\item The cycles of $f_3\circ f_1$ give the succession of edges around the white vertices. If $f_3\circ f_1 \in \mathcal{C}_{\lambda\lambda}$ then the degree distribution of the white vertices is $\lambda$ (counting only once each pair of half edges belonging to the same edge),
\item The cycles of $f_3\circ f_2$ give the succession of edges around the black vertices. If $f_3\circ f_2 \in \mathcal{C}_{\mu\mu}$ then the degree distribution of the black vertices is $\mu$ (counting only once each pair of half edges belonging to the same edge),
\item The cycles of $f_1\circ f_2$ encode the faces of the map. If $f_1\circ f_2 \in \mathcal{C}_{\nu\nu}$ then the degree distribution of the faces is $\nu$
\end{compactitem}
In what follows, we consider the number  $L_{\lambda, \mu}^{n}$ of {\bf unicellular}, or one-face, locally orientable hypermaps with face distribution $\nu=n^1=(n)$, white vertex distribution $\lambda$, and black vertex distribution $\mu$.

Let $f_1$ be the pairing $(1\,\widehat{n})(2\,\widehat{1})(3\,\widehat{2})\ldots(n\,n\widehat{-}1)$ and $f_2 =f_{\star}= (1\,\widehat{1})(2\,\widehat{2})\ldots(n\,\widehat{n})$. We have $f_1\circ f_2 = (123\ldots n)(\widehat{n}n\widehat{-}1\, n\widehat{-}2\ldots\widehat{1}) \in \mathcal{C}_{(n),(n)}$. Then by the above description, one can see that
\[
L_{\lambda, \mu}^{n} =\,\, \mid \{ f_3 \mbox{ pairings in } S_{2n}([n]\cup [\widehat{n}]) \mid f_3\circ f_1 \in \mathcal{C}_{\lambda\lambda}, f_3\circ f_2 \in \mathcal{C}_{\mu\mu} \} \mid.
\]
Moreover, the following relation between $L^n_{\la,\mu}$ and $b^n_{\la,\mu}$ holds: \cite[Cor 2.3]{GJ1}
\begin{equation}
\label{eq:lb}
L_{\lambda, \mu}^{n} = \frac{1}{2^nn!}b_{\lambda, \mu}^{n}.
\end{equation}
Thus we can encode the connection coefficients as numbers of locally orientable  hypermaps and 
\begin{equation} \label{eq:psitermsL}
\Psi^{(n)}=\sum_{\la,\mu\vdash n} L^n_{\la,\mu} p_{\la}p_{\mu}.
\end{equation}
%Let $f_1$ be the pairing $f_{\star}=(1\,\widehat{n})(2\,\widehat{1})(3\,\widehat{2})\ldots(n\,n\widehat{-}1)$ and $f_2 = (1\,\widehat{1})(2\,\widehat{2})\ldots(n\,\widehat{n})$. We have $f_1\circ f_2 = (123\ldots n)(\widehat{n}n\widehat{-}1\, n\widehat{-}2\ldots\widehat{1}) \in \mathcal{C}_{(n),(n)}$. 
%Using  \cite{GJ1}, one can easily show the following relation:
%\begin{equation}
%\label{eq:lb}
%L_{\lambda, \mu}^{n} = \frac{1}{2^nn!}b_{\lambda, \mu}^{(n)}
%\end{equation}
%Furthermore, following Lemma $3.2$ of \cite{HSS}, we have:
%\begin{equation}
%L_{\lambda, \mu}^{n} = \mid \{ f_3 \mbox{ pairing such that } f_3\circ f_1 \in \mathcal{C}_{\lambda,\lambda} \mbox{ and } f_3\circ f_2 \in \mathcal{C}_{\mu,\mu} \} \mid
%\end{equation}

\begin{example}
Figure \ref{fig:example} depicts a locally orientable unicellular hypermap in $L_{\lambda, \mu}^{n}$ with $\lambda=1^12^23^14^1$ and $\mu = 3^14^15^1$ (do not pay attention to the geometric shapes at this stage) where
 where
\begin{align*}
f_3 &= (1\,\widehat{3})(2\,7)(3\,\widehat{10})(4\,12)(5\,\widehat{9})(6\,10)(8\,\widehat{12})(9\,\widehat{8})(11\,\widehat{2})(\widehat{1}\,\,\widehat{6})(\widehat{4}\,\,\widehat{5})(\widehat{7}\,\,\widehat{11}),\\
f_3\circ f_1 &= (\,\widehat{8}\,)(9)\,\,(2\,\widehat{6})(7\,\widehat{1})\,\,(3\,\,11)(\widehat{10}\,\,\widehat{2})\,\,(5\,\widehat{5}\,\widehat{10})(\widehat{9}\,6\,\,\widehat{4})\,\,(1\,8\,\widehat{11}\,4)(\widehat{3}\,12\,\widehat{7}\,\widehat{12}),\\
f_3\circ f_2 &= (2\,11\,\widehat{7})(7\,\widehat{11}\,\widehat{2})\,\,(1 \widehat{6}\,10\,3)(\widehat{3}\,\widehat{10}\,6\,\widehat{1})\,\,(4\,\widehat{5}\,\,\widehat{9}\,\,\widehat{8}\,\,\widehat{12})(12\,8\,9\,5\,\widehat{4}).
\end{align*}

\end{example}

\begin{figure}[htbp]
  \begin{center}
    \includegraphics[width=0.35\textwidth]{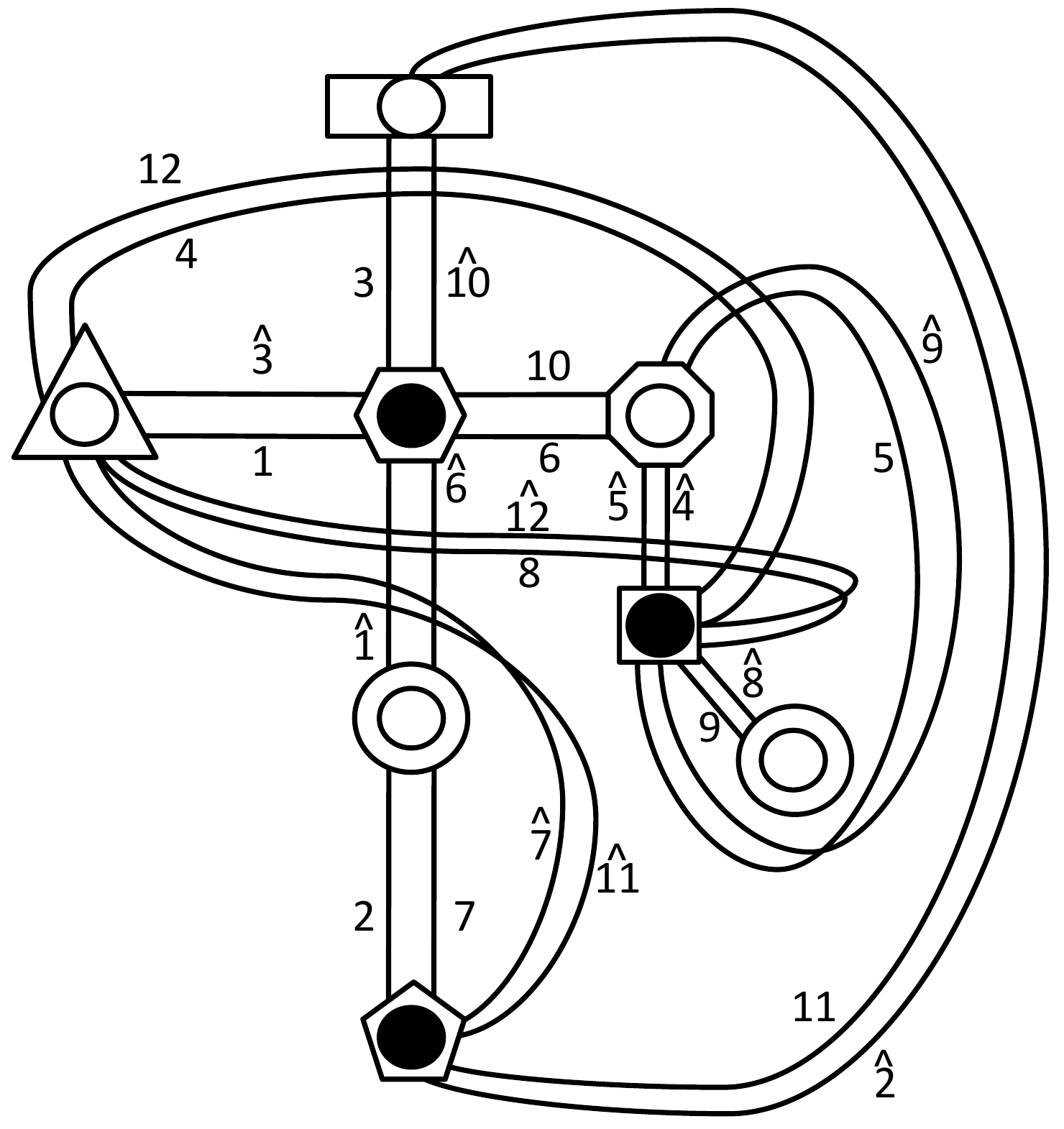}
    \caption{A unicellular locally orientable hypermap}
    \label{fig:example}
  \end{center}
\end{figure}
\subsection{Locally orientable partitioned hypermaps}
Next, we consider locally orientable hypermaps where we partition the white vertices (black resp.). In terms of the pairings, this means we ``color'' the cycles of $f_3\circ f_1$ ($f_3\circ f_2$ resp.) allowing repeated colors but imposing that the two cycles corresponding to each white (black resp.) vertex have the same color. Definition \ref{def:phm} below makes this more precise. 

We will use $\pi$ to denote a set partition of $[n]\cup[\widehat{n}]$ with blocks $\{\pi^1,\pi^2,\ldots,\pi^m\}$. We will only work with set partitions whose blocks have even size. We say that these set partitions $\pi$ have {\bf half-type}, $\frac12$-$\ty$, $\la$ if the cardinalities of the blocks are given by a the integer partition $2\la$ of $2n$.
  
\begin{definition}[Locally orientable partitioned hypermaps] \label{def:phm}
We consider the set of triples, $\mathcal{LP}_{\lambda,\mu}^{n} = (f_3,\pi_1,\pi_2)$ where $f_3$ is a pairing on $[n]\cup [\widehat{n}]$, $\pi_1$ and $\pi_2$ are set partitions on  $[n]\cup [\widehat{n}]$ with blocks of even size and of respective half-types, $\frac12$-$\ty$, $\lambda$ and $\mu$,  with the constraint that $\pi_i$ $(i=1,2)$ is stable by $f_i$ and $f_3$.  Any such triple is called a {\bf locally orientable partitioned hypermap} of type $(\lambda,\mu)$. Furthermore, let $LP_{\lambda, \mu}^{n} = |\mathcal{LP}_{\lambda,\mu}^{n}|$.
\end{definition}

\begin{remark}
The analogous notion of partitioned or colored map is common in the study of orientable maps \eg, see \cite{L},\cite{GN}. Recently, Bernardi \cite{B} extended the approach in \cite{L} to find a correspondence between locally orientable partitioned maps and orientable maps that have a certain distinguished spanning subgraph (called a {\em planar-rooted map}). This correspondence does not fully apply to count locally orientable hypermaps (locally orientable bipartite maps \cite[Sect. 7]{B}).
\end{remark}

\begin{lemma} The number of hat numbers in a block is equal to the number of non hat numbers\end{lemma}
\begin{proof} If non hat number $i$ belongs to block $\pi_1^k$ then $f_1(i) = i\widehat{-}1$ belongs as well to $\pi_1^k$. Same argument applies to blocks of $\pi_2$ with $f_2(i)=\widehat{i}$.\end{proof}
\begin{example}
As an example, the locally orientable hypermap on Figure \ref{fig:example} is partitioned by the blocks: 
\begin{align*}
\nonumber \pi_1 &= \{\{\widehat{12},1,\widehat{3},4,\widehat{7},8,\widehat{11},12\},\,\{\widehat{1},2,\widehat{6},7,\widehat{8},9\},\,\{\widehat{2},3,\widehat{10},11\},\,\{\widehat{4},5,\widehat{5},6,\widehat{9},10\}\}\\
\nonumber \pi_2 &= \{\{1,\widehat{1},3,\widehat{3},6,\widehat{6},10,\widehat{10}\},\phantom{\{}\{2,\widehat{2},7,\widehat{7},11,\widehat{11}\},\phantom{\{}\{4,\widehat{4},5,\widehat{5},8,\widehat{8},9,\widehat{9},12,\widehat{12}\}\}
\end{align*}
(blocks are depicted by the geometric shape, all the vertices belonging to a block have the same shape).
\end{example}

Let $\overline{R}_{\lambda,\mu}$ be the number of unordered partitions $\pi=\{\pi^{1}, \ldots, \pi^{p}\}$ of the set $[\ell(\lambda)]$ such that $\mu_j = \sum_{i\in \pi^j} \lambda_i$ for $1\leq j \leq \ell(\mu)$. Then for the monomial and power symmetric functions, $m_{\lambda}$ and $p_{\lambda}$, we have: $p_{\lambda} = \sum_{\mu \succeq \lambda} Aut(\mu) \overline{R}_{\lambda,\mu} m_{\mu}$ \cite[Prop.7.7.1]{EC2}. We use this to obtain a relation between $L_{\lambda,\mu}^n$ and $LP_{\lambda,\mu}^n$
 
\begin{proposition} \label{prop1} For partitions $\rho, \epsilon \vdash n$ we have  $LP_{\nu,\rho}^n =  \sum_{\lambda,\mu} \overline{R}_{\lambda\nu}\overline{R}_{\mu\rho}L_{\lambda,\mu}^n$, where  $\lambda$ and $\mu$ are refinements of $\nu$ and $\rho$ respectively.
\end{proposition}

\begin{proof} 
Let $(f_3,\pi_1,\pi_2) \in \mathcal{LP}_{\nu,\rho}^n$. If $f_3\circ f_1 \in \mathcal{C}_{\lambda\lambda}$ and $f_3\circ f_2 \in \mathcal{C}_{\mu\mu}$  then by definition of the set partitions we have that $\lambda$ and $\mu$ are refinements of  $\frac12$-$\ty(\pi_1)=\nu$ and $\frac12$-$\ty(\pi_2)=\rho$ respectively. Thus, we can classify the elements of $\mathcal{LP}_{\nu,\rho}^n$ by the cycle types of $f_3\circ f_1$ and $f_3\circ f_2$.  \ie,  $\mathcal{LP}_{\nu,\rho}^n=\bigcup_{\lambda,\mu} \mathcal{LP}_{\nu,\rho}^n(\lambda,\mu)$ where
$$
\mathcal{LP}_{\nu,\rho}(\lambda,\mu) = \{ (f_3,\pi_1,\pi_2) \in \mathcal{LP}^n_{\nu,\rho} ~|~ (f_3\circ f_1,f_3\circ f_2) \in \mathcal{C}_{\lambda\lambda} \times \mathcal{C}_{\mu\mu}\}.
$$
If $LP_{\mu\rho}^n(\lambda,\mu) = |\mathcal{LP}_{\mu\rho}^n(\lambda,\mu)|$ then it is easy to see that $LP_{\mu,\rho}^n(\lambda,\mu)= \overline{R}_{\lambda\nu}\overline{R}_{\mu\rho} L_{\lambda\mu}^{n}$.

%Let $(\pi_1, \pi_2,\pi_3,\al_1,\al_2) \in \mathcal{C}(\rho,\delta,\epsilon)$. If $\al_1 \in \mathcal{C}_{\la}$, $\al_2 \in \mathcal{C}_{\mu}$ and $\al_3 = \al_2^{-1}\al_1^{-1}\ga_n \in \mathcal{C}_{\nu}$  then by the definition of the partitioned cacti we have that $\ty(\pi_1)=\rho \succeq \la$, $\ty(\pi_2)=\delta \succeq \mu$ and $\ty(\pi_3)=\epsilon \succeq \nu$. Thus, if we further refine $C(\rho,\delta,\epsilon)$ by the cycle types of the permutations, {\em i.e.} if
%$$
%\mathcal{C}_{\la,\mu,\nu}(\rho,\delta,\epsilon) = \{ (\pi_1,\pi_2,\pi_3,\al_1,\al_2) \in \mathcal{C}(\rho,\delta,\epsilon) ~|~ (\al_1,\al_2,\al_2^{-1}\al_1^{-1}\ga_n) \in \mathcal{C}_{\la} \times \mathcal{C}_{\mu} \times \mathcal{C}_{\nu}\},
%$$
%then $\mathcal{C}(\rho,\delta,\epsilon) = \bigcup_{\la \preceq \rho, \mu \preceq \delta, \nu \preceq \epsilon} \mathcal{C}_{\la,\mu,\nu}(\rho,\delta,\epsilon) $ where the union is disjoint. Finally, if $C_{\la,\mu,\nu}(\rho,\delta,\epsilon) = |\mathcal{C}_{\la,\mu,\nu}(\rho,\delta,\epsilon)|$ then it is easy to see that $C_{\la,\mu,\nu}(\rho,\delta,\epsilon)=  \ov{R}_{\la\rho}\ov{R}_{\mu\delta}\ov{R}_{\nu\epsilon}k_{\la,\mu,\nu}^n$. 
\end{proof}

By the change of basis equation between $p_{\lambda}$ and $m_{\lambda}$, this immediately relates the generating series $\Psi^{n}$ and the generating series for $LP^n_{\lambda,\mu}$ in  monomial symmetric functions, 
\begin{equation} \label{lem:lp}
\sum_{\lambda,\mu \vdash n}L_{\lambda, \mu}^{n}p_{\lambda}({\bf x})p_{\mu}({\bf y}) = \sum_{\lambda,\mu \vdash n}Aut(\lambda)Aut(\mu)LP_{\lambda, \mu}^{n}m_{\lambda}({\bf x})m_{\mu}({\bf y}).
\end{equation}

\begin{definition}
Let  $\mathcal{LP}({\bf A})$ be the set, with cardinality $LP({\bf A})$, of locally orientable partitioned hypermaps with $n$ edges where ${\bf A}= (P,P',Q,Q')$ are tridimensional arrays such that for $i,j,k \geq 0$:
\begin{compactitem}
\item $P_{ijk}$ (resp. $P'_{ijk}$)  is the number of blocks of $\pi_1$ of half size $i$ such that: 
\begin{compactenum}
\item[(i)] either $1$ belongs to the block or its maximum {\bf non-hat} number is paired to a {\bf hat} number by $f_3$  (resp. blocks of $\pi_1$ not containing $1$ such that the maximum {\bf non-hat} number of the block is paired to a {\bf non-hat} number by $f_3$), 
\item[(ii)] the block contains $j$ pairs $\{t,f_3(t)\}$ where $t$ is the maximum {\bf hat} number of a block of $\pi_2$ such that $f_3(t)$ is also a {\bf hat} number, and, 
\item[(iii)] the block contains as a whole $j+k$ pairs $\{u,f_3(u)\}$ where both $u$ and $f_3(u)$ are {\bf non-hat} numbers.
\end{compactenum}

\item $Q_{ijk}$ (resp. $Q'_{ijk}$) is the number of blocks of $\pi_2$ of half size $i$ such that:
\begin{compactenum}
\item[(i)] the maximum {\bf hat} number of the block is paired to a {\bf non-hat} (resp. {\bf hat}) number by $f_3$,
\item[(ii)] the block contains $j$ pairs $\{t,f_3(t)\}$ where $t$ is the maximum {\bf non-hat} number of a block of $\pi_1$ non containing $1$ and such that $f_3(t)$ is also a {\bf non-hat} number, and
\item[(iii)] the block contains as a whole $j+k$ pairs $\{u,f_3(u)\}$ where both $u$ and $f_3(u)$ are {\bf hat} numbers.
\end{compactenum}
\end{compactitem}
\end{definition}
As a direct consequence, for $LP({\bf A})$ to be non zero ${\bf A}$ has to check the conditions of Theorem \ref{thm:main}. Furthermore:

\begin{equation}\label{eq:lpa}{LP}_{\lambda,\mu}^{n} = \sum_{{\bf A} \in M(\lambda,\mu)}LP({\bf A})\end{equation}
\begin{example}
The partitioned hypermap on Figure \ref{fig:example} belongs to $\mathcal{LP}({\bf A})$ for $P = E_{4,1,0} + E_{3,0,1}+ E_{2,0,0}$, $P' = E_{3,0,1}$, $Q= E_{5,0,1}+E_{4,1,0}$, $Q' = E_{3,0,1}$ where $E_{t,u,v}$, the elementary tridimensional array with $E_{t,u,v}=1$ and $0$ elsewhere.
\end{example}
One can notice that a hypermap is orientable if and only if $f_3(t)$ is a hat number when $t$ is a non hat number (we go through each edge of the map in both directions and there are no changes of direction during the map traversal). As a result, a hypermap in $\mathcal{LP}({\bf A})$ is orientable if and only if:
\begin{compactitem}
\item $p'=q'=r=0$ and
\item $P_{ijk} = Q_{ijk}=0$ if $j>0$ and/or $k>0$.
\end{compactitem}
In this particular case, we have the following values for $\mathcal{N}({\bf A})$ and ${\bf A!}$ :
\begin{compactitem}
\item $\mathcal{N}({\bf A}) = \sum_{i,i,k}iP_{ijk}=\sum_{t}i n_i(\lambda)=n$
\item ${\bf A!}=\prod_{i}P_{i,0,0}!Q_{i,0,0}!=Aut(\lambda)Aut(\mu)$
\end{compactitem}
If we denote $c_{\lambda,\mu}^n$ the number of such orientable maps, by Theorem \ref{thm:main}, Equation \eqref{eq:lpa}, Lemma \ref{lem:lp} and Relation \eqref{eq:lb} we recover the following result from \cite[Thm. 1]{MV}:

\begin{theorem*}
\begin{equation} \label{eq:orientable}
\sum_{\lambda,\mu \vdash n} c_{\lambda,\mu}^np_\lambda p_\mu = n\sum_{\lambda,\mu \vdash n}\frac{(n-\ell(\lambda))!(n-\ell(\mu))!}{(n+1-\ell(\lambda)-\ell(\mu))!}m_\lambda m_\mu
\end{equation}
\end{theorem*}
Note that $\{c_{\lambda,\mu}^n\}$ are as well connection coefficients of the symmetric group, that count the number of ordered factorizations $w_1\cdot w_2$ of the long cycle $(1,2,\ldots,n)$ in $S_n$ where $w_1\in \mathcal{C}_{\la}$ and $w_2\in \mathcal{C}_{\mu}$. 

\subsection{Permuted forests and reformulation of the main theorem}
We show that locally orientable partitioned hypermaps admit a nice bijective interpretation in terms of some recursive forests defined as follows:
\begin{definition}[Rooted bicolored forests of degree {\bf A}]\label{def:forest} In what follow we consider the set $\mathcal{F}({\bf A})$ of permuted rooted forests composed of:
\begin{compactitem}
\item a bicolored identified ordered {\bf seed tree} with a white root vertex,
\item other bicolored ordered trees, called {\bf non-seed trees}  with either a white or a black root vertex,
\item each vertex of the forest has three kinds of ordered descendants : {\bf tree-edges} or edges (connecting a white and a black vertex), {\bf thorns} (halves of tree-edges connected to only one vertex) and {\bf loops} connecting a vertex to itself. The two {\em extremities} of the loop are part of the ordered set of descendants of the incident vertex and therefore the loop can be intersected by thorns, edges and other loops as well.
\end{compactitem}
The forests in $\mathcal{F}({\bf A})$ also have the following properties:
\begin{compactenum}[(i)]
\item the root vertices of the non-seed trees have at least one descending loop with one extremity being the rightmost descendant of the considered vertex,
\item the total number of thorns (resp. loops) connected to the white vertices is equal to the number of thorns (resp. loops) connected to the black vertices,
%\item the total number of loops connected to the white vertices is equal to the number of loops connected to the black ones,
\item there is a bijection between thorns connected to white vertices and the thorns connected to black vertices. The bijection between thorns will be encoded by assigning the same symbolic {\em latin} labels $\{a,b,c,\ldots\}$ to thorns associated by this bijection,
\item there is a mapping that associates to each loop incident to a white (resp. black) vertex, a black (resp. white) vertex ${\rm v}$  such that the number of white (resp. black) loops associated to a fixed black (resp. white) vertex ${\rm v}$ is equal to its number of incident loops. We will use symbolic {\em greek} labels $\{\alpha,\beta,\ldots\}$ to associate loops with vertices except for the maximal loop of a root vertex ${\rm r}$ of the non-seed trees. In this case, we draw an arrow (\begin{tikzpicture} \draw[very thick,densely dashed] [->] (0,0) to (0.8,0);   \end{tikzpicture}) outgoing from the root vertex ${\rm r}$ and incoming to the vertex associated with the loop. Arrows are non ordered, and
\item the ascendant/descendant structure defined by the edges of the forest and the arrows defined above is a tree structure rooted in the root of the seed tree.
\end{compactenum}

Finally the degree ${\bf A}$ of the forest is given in the following way:
\begin{compactenum}
\item[(vii)] $P_{ijk}$ (resp $P'_{ijk}$) counts the number of non root white vertices (including the root of the seed tree) (resp. white root vertices excluding the root of the seed tree) of degree $i$, with $j$ incoming arrows and a total of $j+k$ loops.
\item[(viii)] $Q_{ijk}$ (resp $Q'_{ijk}$) counts the number of non root black vertices (resp. black root vertices) of degree $i$, with $j$ incoming arrows and total $j+k$ loops.
\end{compactenum}
\end{definition}

\begin{example}
As an example, Figure \ref{fig:exrecons} depicts two permuted forests. The one on the left is of degree ${\bf A}=(P,P',Q,Q')$ for $P = E_{4,1,0} + E_{3,0,1}+E_{2,0,0}$, $P' = E_{3,0,1}$, $Q= E_{5,0,1} + E_{4,1,0}$, and  $Q' = E_{3,0,1}$ while the one on the right is of degree ${\bf A^{(2)}}=(P^{(2)},P'^{(2)},Q^{(2)},Q'^{(2)})$ for $P^{(2)}= E_{7,0,3} + E_{4,1,0}$, $P'^{(2)}=\{0\}_{i,j,k}$, $Q^{(2)}=E_{7,0,2}$, and $Q'^{(2)}=E_{4,0,2}$.

\begin{figure}[htbp]
  \begin{center}
    \includegraphics[width=60mm]{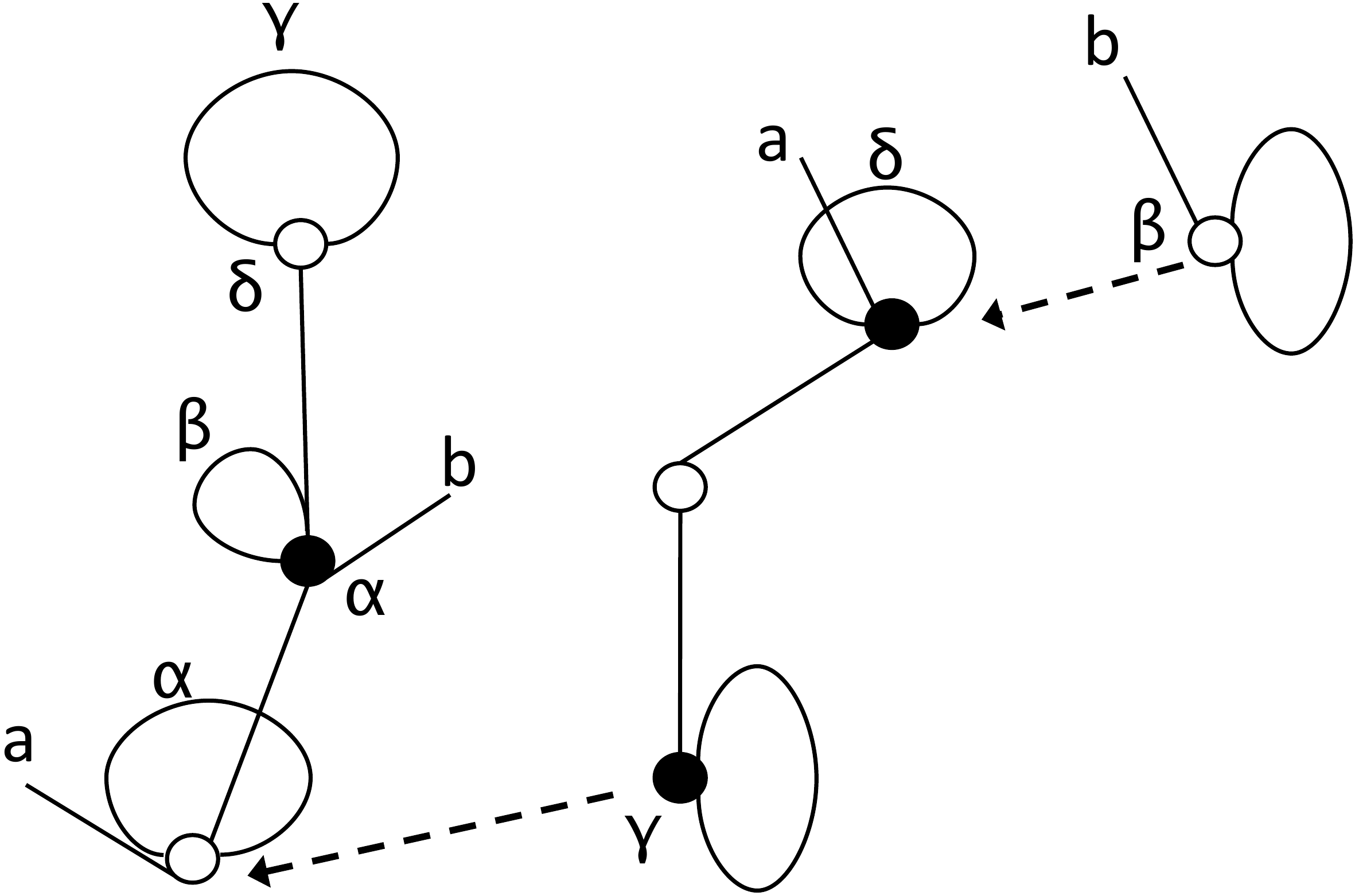}\hspace{15mm}
    \includegraphics[width=30mm]{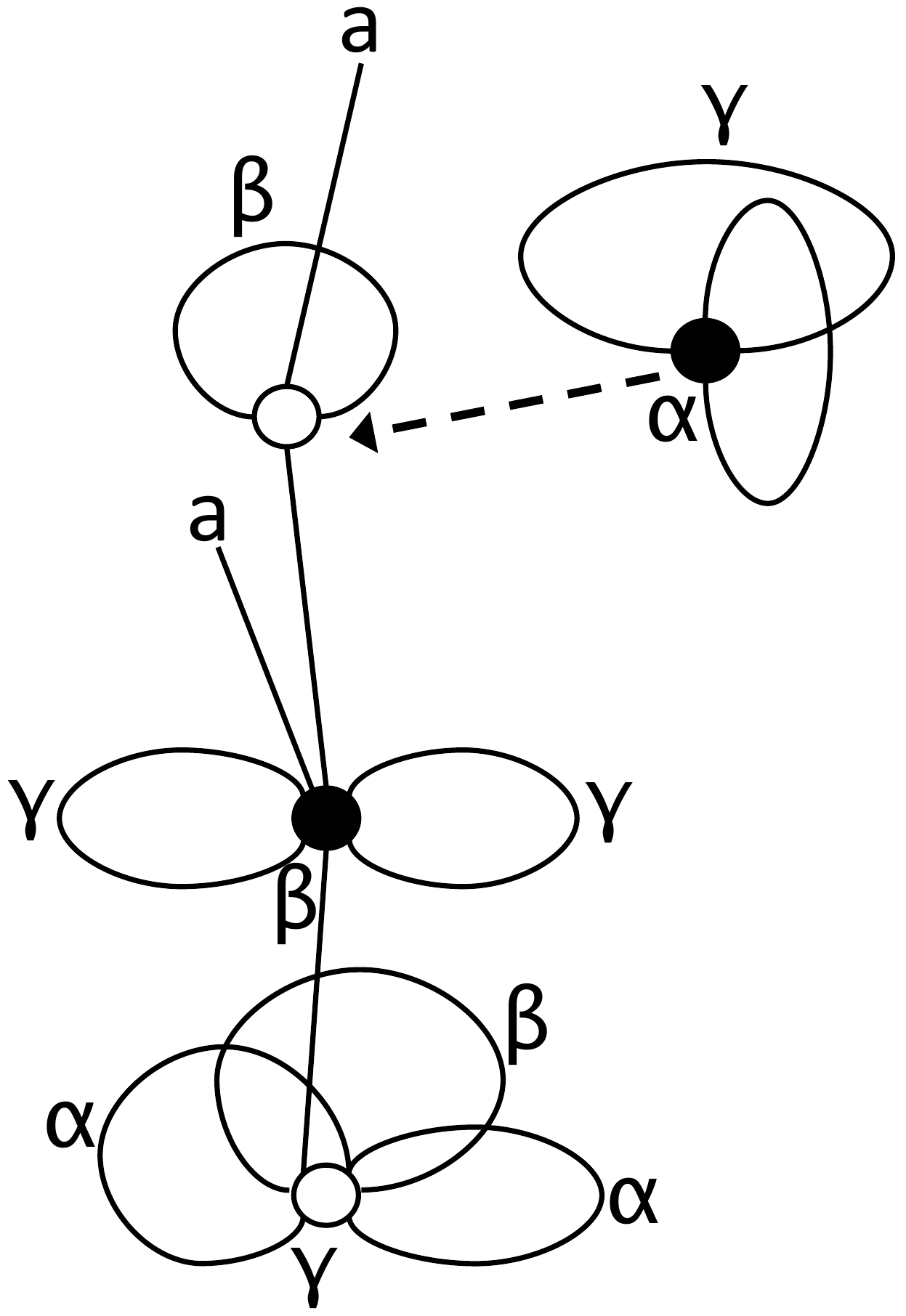}
    \caption{Two Permuted Forests}
    \label{fig:exrecons}
  \end{center}
\end{figure}
\end{example}

\begin{lemma}
Using the Lagrange theorem for implicit functions one can show:
\footnotesize
\begin{align*}
F({\bf A}) = \frac{\mathcal{N}({\bf A})}{{\bf A!}}\frac{(n-q-2r)!(n-p-2r)!}{(n+1-p-q-2r)!}\frac{p'!q'!\left ( r - p' \right)!\left ( r- q' \right)!}{2^{2r-p'-q'}}\prod_{i,j,k}{\binom{i-1}{j,k,j+k}}^{(P+Q)_{ijk}}{\binom{i-1}{j,k,j+k-1}}^{(P'+Q')_{ijk}} 
\end{align*}
\normalsize
\end{lemma}
{\bf Reformulation of the main theorem}\\
The next sections are dedicated to the proof of the following stronger result that will imply Theorem \eqref{thm:main}:
\begin{theorem}
There is a bijection $\Theta_{\bf A}:\mathcal{LP}({\bf A}) \to \mathcal{F}({\bf A}$) and so $LP({\bf A}) = F({\bf A})$.
\end{theorem}

\section{Bijection between locally orientable unicellular partitioned  hypermaps and permuted forests}
We proceed with the description of the bijective mapping $\Theta_{\bf A}$ between locally orientable partitioned hypermaps and permuted forests of degree ${\bf A}$. Let $(f_3,\pi_1,\pi_2)$ be a partitioned hypermap in $\mathcal{LP}({\bf A})$. The
first step is to define a set of white and black vertices with labeled ordered half edges such that:
\begin{compactitem}
\item each white vertex is associated to a block of $\pi_1$ and each black vertex is associated to a block of $\pi_2$,
\item the number of half edges connected to a vertex is half the cardinality of the associated block, and
\item the half edges connected to the white (resp. black) vertices are labeled with the non hat (resp. hat) integers in the associated blocks so that moving clockwise around the vertices the integers are sorted in increasing order.
\end{compactitem}
Then we define an ascendant/descendant structure on the vertices. A black vertex $b$ is the descendant of a white one $w$ if the maximum half edge label of $b$ belongs to the block of $\pi_1$ associated to $w$. Similar rules apply to define the ascendant of each white vertex except the one containing the half edge with label $1$.

If black vertex $b^d$ (resp. white vertex $w^d$) is a descendant of white vertex $w^a$ (resp. black vertex $b^a$) and has maximum half edge label $m$ such that $f_3(m)$ is the label of a half edge of $w^a$ (resp. $b^a$), i.e. $f_3(m^b)$ is a non hat (resp. hat) number, then we connect these two half edges to form an edge. Otherwise $f_3(m)$ is a hat (resp. non hat) number and we draw an arrow (\begin{tikzpicture} \draw[very thick,densely dashed] [->] (0,0) to (0.8,0);   \end{tikzpicture}) between the two vertices. Note that descending edges are ordered but arrows are not.\\ 
\begin{lemma} \label{lem:tree}The above construction defines a tree structure rooted in the white vertex with half edge $1$. \end{lemma}
\begin{proof}
Let black vertices $b_1$ and $b_2$ associated to blocks $\pi_2^{b_1}$ and $\pi_2^{b_2}$ be respectively a descendant and the ascendant of white vertex $w$ associated to $\pi_1^w$. We denote by $m^{b_1}$, $m^{b_2}$ and $m^{w}$ their respective maximum half edge labels (hat, hat, and non hat) and assume $m^{b_1} \neq \widehat{n}$. As $\pi_1^{w}$ is stable by $f_1$, then $f_1(m^{b_1})$ is a non hat number in $\pi_1^w$ not equal to $1$. It follows $m^{b_1} < f_1(m^{b_1}) \leq m^w < f_2(m^w)$. Then as  $\pi_2^{b_2}$ is stable by $f_2$, it contains $f_2(m^w)$ and $f_2(m^w) \leq m^{b_2}$. Putting everything together yields $m^{b_1} < m^{b_2}$. In a similar fashion, assume white vertices $w_1$ and $w_2$ are descendant and ascendant of black vertex $b$. If we note $m^{w_1}$, $m^{w_2}$ and $m^{b}$ their maximum half edge labels (non hat, non hat, and hat) with $m^{b} \neq \widehat{n}$, one can show that $m^{w_1} < m^{w_2}$. Finally, as $f_1(\widehat{n}) =1$, the black vertex with maximum half edge $\widehat{n}$ is descendant of the white vertex containing the half edge label $1$. \end{proof}
\begin{example}
Using the hypermap of Figure \ref{fig:example} we get the set of vertices and ascendant/descendant structure as described on Figure \ref{fig:verttree}.
\begin{figure}[htbp]
  \begin{center}
    \includegraphics[width=0.9\textwidth]{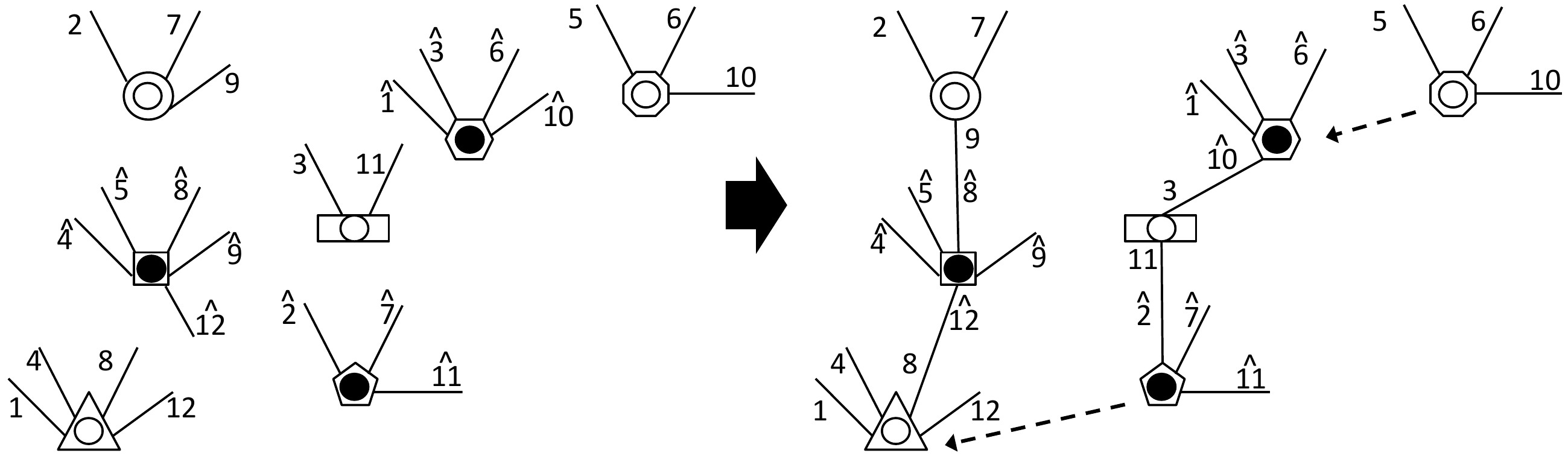}
    \caption{Construction of the ascendant/descendant structure}
    \label{fig:verttree}
  \end{center}
\end{figure}
\end{example}
Next, we proceed by linking half edges connected to the same vertex if their labels are paired by $f_3$ to form loops. If $i$ and $f_3(i)$ are the labels of a loop connected to a white (resp. black) vertex such that neither $i$ nor $f_3(i)$ are maximum labels (except if the vertex is the root), we assign the same label to the loop and the black (resp. white) vertex associated to the block of $\pi_2$ (resp. $\pi_1$) also containing $i$ and $f_3(i)$. As we assign at most one label from $\{\alpha,\beta,\ldots\}$ to a given vertex, several loops may share the same label. 
\begin{lemma}\label{lem:loops} The number of loops connected to the vertex labeled $\alpha$ is equal to its number of incoming arrows plus the number of loops  labeled $\alpha$ incident to other vertices in the forest.
 \end{lemma}
\begin{proof} The result is a direct consequence of the fact that in each block the number of hat/hat pairs is equal to the number of non hat/non hat pairs\end{proof} 

As a final step, we define a permutation between the remaining half edges (thorns) connected to the white vertices and the one connected to the black vertices. If two remaining thorns are paired by $f_3$ then these two thorns are given the same label from $\{a,b,\ldots\}$.  All the original integer labels are then removed.
We denote by $\widetilde{F}$ the resulting forest.
\begin{example}
We continue with the hypermap from Figure \ref{fig:example} and perform the final steps of the construction as described on Figure \ref{fig:loopsperm}.(Note that the geometric shapes are here for reference only, they do not play any role in the final object $\widetilde{F}$).
\begin{figure}[htbp]
  \begin{center}
    \includegraphics[width=0.9\textwidth]{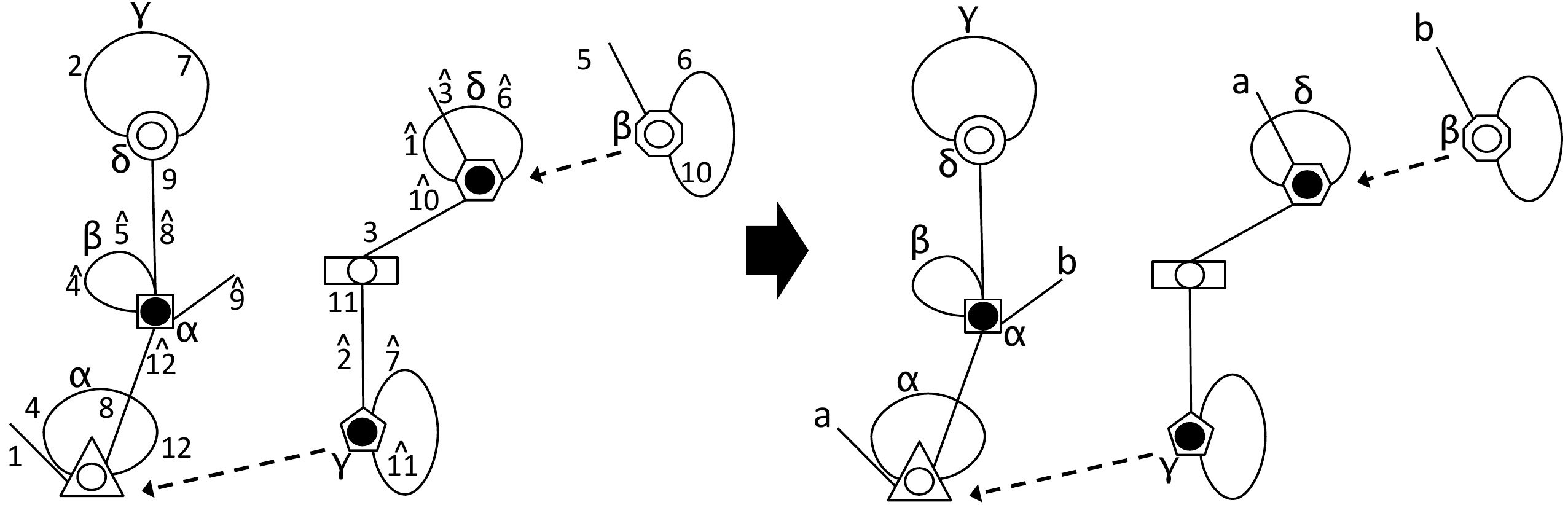}
    \caption{Final steps of the permuted forest construction}
    \label{fig:loopsperm}
  \end{center}
\end{figure}
\end{example}

\noindent As a direct consequence of Definition \ref{def:forest}, $\widetilde{F}$ belongs to $\mathcal{F}(A)$.

\section{Proof of the bijection}
We show that mapping $\Theta_{\bf A}:(f_3,\pi_1,\pi_2) \mapsto \widetilde{F}$ is indeed one-to-one.
\subsection{Injectivity}
We start with a forest $\widetilde{F}$ in $\mathcal{F}({\bf A})$ and show that there is at most one triple $(f_3,\pi_1,\pi_2)$ in $\mathcal{LP}({\bf A})$ such that $\Theta_{\bf A}(f_3,\pi_1,\pi_2) = \widetilde{F}$. The first part is to notice that within the construction in $\Theta_{\bf A}$, the original integer label of the leftmost descendant (thorn, half loop or edge) of the root vertex of the seed tree is necessarily $1$ (this root is the vertex containing $1$ and the labels are sorted in increasing order from left to right).\\
Assume we have recovered the positions of integer labels $1,\widehat{1},2,\widehat{2},\ldots,i$, for some $1 \leq i \leq n-1$, non hat number. Then four cases can occur:
\begin{compactitem}
\item $i$ is the integer label of a thorn of latin label $a$. In this case, $f_3(i)$ is necessarily the integer label of the thorn connected to a black vertex also labeled with $a$. But as the blocks of $\pi_2$ are stable by both $f_3$ and $f_2$ then $\widehat{i} = f_2(i)$ is the integer label of one of the descendants of the black vertex with thorn $a$. As these labels are sorted in increasing order, necessarily, $\widehat{i}$ labels the leftmost descendant with no recovered integer label.
\item $i$ is the integer label of a half loop of greek label $\alpha$. Then, in a similar fashion as above $\widehat{i}$ is necessarily the leftmost unrecovered integer label of the black vertex with symbolic label $\alpha$.
\item $i$ is the integer label of a half loop with no symbolic label (i.e, either $i$ or $f_3(i)$ is the maximum label of the considered white vertex). Then, $\widehat{i}$ is necessarily the leftmost unrecovered integer label of the black vertex at the other extremity of the arrow outgoing from the white vertex containing integer label $i$.
\item $i$ is the integer label of an edge. $\widehat{i}$ is necessarily the leftmost unrecovered integer label of the black vertex at the other extremity of this edge.
\end{compactitem}
Finally, using similar four cases for the black vertex containing the descendant with integer label $\widehat{i}$ and the fact that blocks of $\pi_1$ are stable by $f_3$ and $f_1$, the thorn, half loop or edge with integer label $i+1 = f_1(\,\widehat{i}\,)$ is uniquely determined as well.

We continue with the procedure described above up until we fully recover all the original labels $[n]\cup [\hat{n}]$. According to the construction of $\widetilde{F}$, the knowledge of all the integer labels uniquely determines the blocks of $\pi_1$ and $\pi_2$. The pairing $f_3$ is as well uniquely determined by the loops, edges and thorns with same latin labels.

\begin{example}
Assume the permuted forest $\widetilde{F}$ is the one on the right hand side of Figure \ref{fig:exrecons}. The steps of the reconstruction are summarized in Figure \ref{fig:recons}. We get that the unique triple $(f_3,\pi_1,\pi_2)$ such that $\Theta_{\bf A}(f_3,\pi_1,\pi_2) = \widetilde{F}$ is:
\begin{align*}
\nonumber f_3 &= (1\,\, 4)(\widehat{1}\,\, \widehat{8})(2\,\, 9)(\widehat{2}\,\, \widehat{3})(3\,\, \widehat{11})(\widehat{4}\,\, \widehat{10})(5\,\, 7)(\widehat{5}\,\, 6)(\widehat{6}\,\, 11)(\widehat{7}\,\, \widehat{9})(8\,\, 10)\\
\nonumber \pi_1 &= \{\{\widehat{11},1,\widehat{1},2,\widehat{2},3,\widehat{3},4,\widehat{7},8,\widehat{8},9,\widehat{9},10\};\{\widehat{4},5,\widehat{5},6,\widehat{6},7,\widehat{10},11\}\}\\
\nonumber \pi_2 &= \{\{2,\widehat{2},3,\widehat{3},5,\widehat{5},6,\widehat{6},7,\widehat{7},9,\widehat{9},11,\widehat{11}\};\{1,\widehat{1},4,\widehat{4},8,\widehat{8},10,\widehat{10}\}\}
\end{align*}
\begin{figure}[htbp]
  \begin{center}
    \includegraphics[width=1\textwidth]{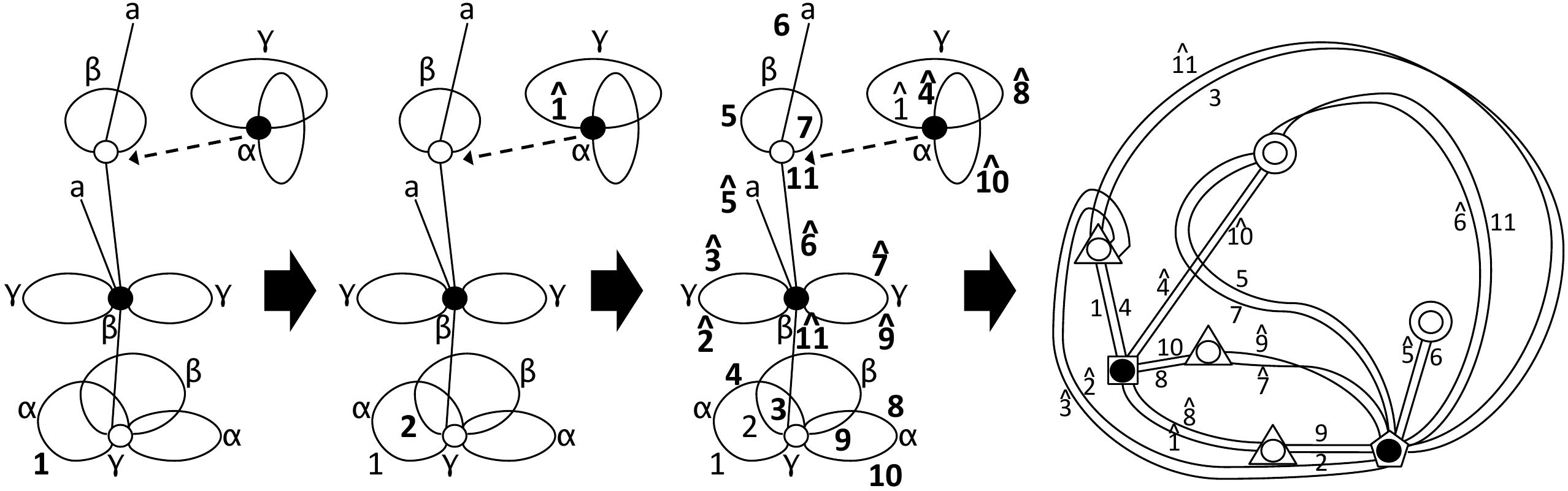}
    \caption{Recovery of the integer labels and the partitioned map}
    \label{fig:recons}
  \end{center}
\end{figure}
\end{example}
\subsection{Surjectivity}
To prove that $\Theta_{\bf A}$ is surjective, we have to show that the reconstruction procedure of the previous section always finishes with a valid output.\\ 
Assume the procedure comes to an end at step $i$ before all the integer labels are recovered (where $i$ is for example non hat, the hat case having a similar proof). It means that we have already recovered all the labels of vertex $v^i$ identified as the one containing $\widehat{i}$ (or $i+1$) prior to this step. This is impossible by construction provided $v^i$ is not the root vertex of the seed tree since the number of times a vertex is identified for the next step is equal to its number of thorns, plus its number of edges, plus twice the number of loops that have the same greek label as $v^i$, plus twice the incoming arrows. Using Lemma \ref{lem:loops}, we have that the sum of the two latter numbers is twice the number of loops of $v^i$. As a consequence, the total number of times the recovering process goes through $v^i$ is exactly (and thus never more than) the degree of $v^i$.

If $v$ is the root vertex of the seed tree, the situation is slightly different due to the fact that we recover label $1$ before we start the procedure. To ensure that the procedure does not terminate prior to its end, we need to show that the $\mid v \mid$-th time the procedure goes through the root vertex is right after all the labels of the forest have been recovered. Again, this is always true because:
\begin{compactitem}
\item The last element of a vertex to be recovered is the label of the maximum element of the associated block. Consequently, all the elements of a vertex are recovered only when all the elements of the descending vertices (through both arrows and edges) are recovered.
\item Lemma \ref{lem:tree} states that the ascendant/descendant structure involving both edges and arrows is a tree rooted in $v$. As a result, the procedure goes the $v$-th time through $v$ only when all the elements of all the other vertices are recovered.
\end{compactitem}  
\section{On proving Theorem \ref{thm:main} using Zonal polynomials}

In the orientable case, one can use Schur symmetric functions and the irreducible characters of the symmetric group to prove the identity in Equation \eqref{eq:orientable} (see \cite{J}). This requires: (i) ($p_{\la}\to s_{\mu}$) the Murnaghan-Nakayama rule \cite[Thm. 7.17.3]{EC2}, (ii) ($p_{\mu}\to m_{\nu}$) finding the number of semistandard Young tableaux  of hook shape $\la=a\, 1^{n-a}$ and type $\mu$ which is just $\binom{\ell(\mu)-1}{a}$, and (iii) using inclusion exclusion. One could try to replicate this on $\Psi^n$ and obtain an algebraic proof of Theorem \ref{thm:main}. We show the outcome after step (i)' using \cite[Cor. 5.2]{HSS}. Steps (ii)' and (iii)' appear quite less tractable.
\begin{compactenum}
\item[(i)'] $\Psi^n({\bf x},{\bf y}) = \frac{|B_n|}{|K_{(n)}|}\sum_{a\geq b \geq 1}\frac{ \varphi^{(a,b,1^{n-a-b})}(n)}{H_{2(a,b,1^{n-a-b})}} Z_{(a,b,1^{n-a-b})}({\bf x})Z_{(a,b,1^{n-a-b})}({\bf y})$.
\end{compactenum}
\section{Appendix : special cases of the main formula} \label{app}
For the formula of Theorem \ref{thm:main} to be complete, we need to observe the following cases:
\begin{compactenum}
\item[(i)] If $q'\neq 0$, then there is at most one given $(i_0,j_0,k_0)$ with $i_0 = 2(j_0+k_0)$, for which $P_{i_0j_0k_0}=1$ instead of $0$. In that situation, we have  
 \footnotesize 
\begin{multline*}
\mathcal{N}({\bf A})\binom{i_0-1}{j_0,k_0,j_0+k_0}^{P_{i_0j_0k_0}}  = \\
 \frac{j_0}{q'}\binom{i_0}{j_0,k_0} \left (\frac{\delta_{p'\neq 0}}{p'}\sum_{i,j,k}{(i-2j-2k)P'}\sum_{i,j,k}{jQ'}+ \sum_{i,j,k}{(i-2j-2k)Q'}\frac{1+\sum_{i,j,k}{(i-1-2j-2k)P}}{n-q-2r}\right )$$
\end{multline*}
\normalsize

\item[(ii)] When $q'\neq 0$ and $n-q-2r=0$, which can only occur if $p=1$ (we assume $P=E_{t,u,v}$, the  elementary tridimensional array with $E_{t,u,v}=1$ and $0$ elsewhere)  the whole summand $\sum_{{\bf A} \in M(\lambda,\mu)}\cdots$ on the RHS of Equation \eqref{eq:mainthm} reduces to:
\footnotesize
\begin{align*}
&\frac{1}{{\bf A!}}\left [ \delta_{(n-2r)p'\neq 0}\frac{(t-2u-2v)\sum_{j,k,l}jQ}{(n-2r)p'}+\delta_{p'\neq 0}u\frac{\sum_{i,j,k}jQ'}{p'q'} + \delta_{p' = 0}\right ]\\
&\times\frac{(n-2r)!p'!q'!\left ( r - p' \right)!\left ( r- q' \right)!}{2^{2r-p'-q'}}\binom{t}{u,v,u+v}\prod_{i,j,l}{\binom{i-1}{j,k,j+k}}^{Q_{ijk}}{\binom{i-1}{j,k,j+k-1}}^{(P'+Q')_{ijk}}
\end{align*}
\normalsize
\end{compactenum}

\bibliographystyle{abbrvnat}
% use the following instead if you encounter problems 
%\bibliographystyle{alpha}
%\bibliography{sample}
%\label{sec:biblio}

\small

\small

\noindent Alejandro H. Morales \\
Department of Mathematics, Massachusetts Institute of Technology \\
Cambridge, MA USA 02139\\
{\tt ahmorales@math.mit.edu}

\bigskip

\noindent Ekaterina A. Vassilieva \\
Laboratoire d'Informatique de l'Ecole Polytechnique\\
91128 Palaiseau Cedex, France\\ 
{\tt katya@lix.polytechnique.fr}

\end{document}